\documentclass[12pt]{amsart}
\usepackage{amssymb}
\usepackage{verbatim}
\usepackage[usenames]{color}
\usepackage{hyperref}
\usepackage{url}
\usepackage[all]{xy}
\newtheorem{theorem}{Theorem}[section]
\newtheorem{proposition}[theorem]{Proposition}

\newtheorem{fact}[theorem]{Fact}

\theoremstyle{definition}
\newtheorem{definition}[theorem]{Definition}

\newtheorem{remark}[theorem]{Remark}

\newtheorem{question}[theorem]{Question}
\newtheorem{observation}[theorem]{Observation}

\theoremstyle{remark}
\newcommand{\defemph}{\textit}

\newcommand{\cf}{\mbox{cf}}

\newcommand{\mc}{\mathcal}
\newcommand{\mbb}{\mathbb}

\newcommand{\forces}{\Vdash}

\newcommand{\Lauchli}{L{\"{a}}uchli}

\DeclareMathOperator{\Dom}{Dom}

\title{Forcing and the Halpern-L\"auchli Theorem}

\author{Natasha Dobrinen}
\address{Department of Mathematics\\
 University of Denver \\
C.M.\ Knudson Hall, Room 300\\
2390 S.\ York St.\\ Denver, CO \ 80208 U.S.A.}
\email{natasha.dobrinen@du.edu}
  \urladdr{\url{http://web.cs.du.edu/~ndobrine}}

\author{Daniel Hathaway}
\address{Department of Mathematics\\
 University of Vermont \\
 16 Colchester Ave. \\
 Burlington, VT \ 05401 U.S.A.}
\email{Daniel.Hathaway@uvm.edu}
 \urladdr{\url{http://mysite.du.edu/~dhathaw2/}}

\subjclass[2010]{Fill in}

\thanks{The first author was partially supported from  National Science Foundation Grant DMS-1600781}

\begin{document}

\begin{abstract}
We investigate the  effects
of various  forcings  on several forms of the
 Halpern-L\"auchli Theorem.
For  inaccessible $\kappa$, we show they
 are preserved by forcings of size  less than $\kappa$.
Combining this with work of
 Zhang in \cite{Zhang17} yields
 that
the polarized partition
 relations associated with  finite products of the $\kappa$-rationals
are   preserved by all forcings
 of size less than $\kappa$ over models satisfying the Halpern-L\"auchli  Theorem at $\kappa$.
We also show that the
 Halpern-L\"auchli Theorem
 is preserved by
 ${{<}\kappa}$-closed forcings
 assuming $\kappa$ is measurable, following  some observed reflection properties.
\end{abstract}

\maketitle

\section{Introduction}

The Halpern-\Lauchli\ Theorem
\cite{Halpern/Lauchli66}
is a Ramsey theorem for products of finitely many trees
 of height $\omega$ which are finitely branching and have no terminal nodes.
It was discovered as a central lemma to the proof in \cite{Halpern/Levy71} that the Boolean Prime Ideal Theorem (the fact that every ideal in a  Boolean algebra can be extended to a prime ideal) is strictly weaker than the Axiom of Choice, over Zermelo-Fraenkel  set theory.
Many variations have been studied, some of which are equivalent to the statement that the BPI is strictly weaker than the Axiom of Choice.
Recent compendia of various
versions of the Halpern-\Lauchli\ Theorem
 appear in \cite{TodorcevicBK10} and
\cite{Dodos/KanBK}.
The Halpern-\Lauchli\ Theorem has found numerous applications in proofs of  partition relations for countable structures, directly to  products of rationals (see \cite{Laver84}) and via the closely related  theorem  of Milliken for strong trees \cite{Milliken79} to
 finite sets of rationals (see \cite{DevlinThesis})
 and the Rado graph (see \cite{Sauer06} and \cite{Laflamme/Sauer/Vuksanovic06}).

Some years after the Halpern-\Lauchli\ Theorem was discovered,
Harrington
found a proof
which uses the techniques and language of forcing, though without actually   passing to a generic extension of the ground model.
Though this proof was well-known in certain circles, a published version did not appear until
\cite{Farah/TodorcevicBK}.

Shelah applied this proof
method of Harrington in
\cite{Shelah91} to  prove a partition theorem (analogue of Milliken's Theorem) for trees on $2^{<\kappa}$, where $\kappa$ is a cardinal whose measurability is preserved by adding many Cohen subsets of $\kappa$.
This  result was extended and applied by D\v{z}amonja, Larson and Mitchell in \cite{Dzamonja/Larson/MitchellQ09} and \cite{Dzamonja/Larson/MitchellRado09}
to obtain partition relations on the $\kappa$-rationals and $\kappa$-Rado graph.
This work, as well as the exposition in Chapter 3 of Todorcevic's book \cite{TodorcevicBK10}
informed the authors' previous work on variations of the Halpern-\Lauchli\ Theorem for more than one tree at uncountable cardinals.
 In \cite{Dobrinen/Hathaway17},
we mapped out the implications between weaker and stronger forms of the Halpern-\Lauchli\ Theorem on trees of uncountable height, and found a better upper bound for the consistency strength of the theorem holding for finitely many trees at a measurable cardinal.
Building on  work in  \cite{Dzamonja/Larson/MitchellQ09} and \cite{Dobrinen/Hathaway17},
Zhang proved a stronger tail-cone version at measurable cardinals and applied it to obtain the analogue of Laver's partition relation for products of finitely many trees on a measurable cardinal (see \cite{Zhang17}).

It is intriguing   that
all  theorems for trees of uncountable height proved so far have  required assumptions beyond ZFC.
In fact, it is still unknown  whether such partition relation theorems for trees at uncountable heights simply are true in ZFC or whether they entail some large cardinal strength.
For more discussion of this main problem and other related open problems, see Section \ref{conclusion}.

In this paper, we are
interested in which forcings preserve the Halpern-L\"{a}uchli Theorem for trees of uncountable height, once it is known to hold.
Section \ref {sec.2} contains basic definitions, most of which are found in \cite{Dobrinen/Hathaway17} and \cite{Zhang17}.
It also contains an equivalence between the tail-cone version of the Halpern-L\"{a}uchli Theorem and a modified version which is easier to satisfy in practice.
Section \ref{section.Derived} contains a new
 method  which constructs a tree in the ground model using forcing names; this is called the {\em derived tree} from a name for a tree.
The Derived Tree Theorem is proved there.

The Derived Tree Theorem
 is applied in Section \ref{sec.4}  to show that small forcings preserve the Halpern-L\"{a}uchli Theorem and its tail-cone version.
As the partition relation on finite products of $\kappa$-rationals holds in any model where  the tail-cone version holds (by work of Zhang in \cite{Zhang17}),
our work shows that this partition
 relation is preserved by any further small forcing.

Section \ref{sec.reflection} presents some instances when the  somewhere dense version (SDHL) has reflection properties.  Thus, if SDHL holds for a stationary set of cardinals below a strongly inaccessible cardinal $\kappa$, then it holds at $\kappa$.
Second, we prove that for a measurable cardinal $\kappa$, SDHL holds at $\kappa$ if and only if the set of cardinals below $\kappa$ where SDHL holds is a member of each normal ultrafilter on $\kappa$.
We apply this to show that ${<}\kappa$-closed forcings preserve SDHL at measurable cardinals.

Finally, in Section \ref{sec.SDHLnotwc}, we provide a model of ZFC where SDHL holds at some regular cardinal which is not weakly compact.  This produces a different  model than the one in \cite{Zhang17}, one that is obtained by a large collection of forcings.
Section \ref{conclusion} contains several key questions
brought to the fore by this work.
Although interesting in their own right, they are all
 sub-problems of the main open problem:
Is   the Halpern-\Lauchli\ Theorem  for trees  of any cardinal height simply true in ZFC?

%%%%%%%%%%%%%%%%%%%%
%%%%%%%%%%%%%%%%%%%%%%
%%%%%%%%%%%%%%%%%%%%

\section{Basic Definitions}\label{sec.2}

We review here some  fundamental
 definitions from  \cite{Dobrinen/Hathaway17}.
%Some of these definitions only make sense when $\kappa$ is strongly inaccessible; in practice there will be no confusion.
Given sequences $s$ and $t$,
 the notation
 $s \sqsubseteq t$ means that $s$
 is an initial segment of $t$;
the notation $s\sqsubset t$  denotes that $s$ is a proper initial segment of $t$.
A set $T \subseteq {^{<\kappa} \kappa}$ is a \textit{tree}
 iff it is closed under taking initial segments.
For $t\in {^{<\kappa} \kappa}$, let $\Dom(t)$ denote the domain of $t$.
We shall also call this the {\em length} of $t$.
 %and often  denote it as $|t|$.
Given  $\alpha \le \Dom(t)$,
 we write $t \restriction \alpha$ for the
 unique initial segment of $t$ of length $\alpha$.

\begin{definition}\label{defn.regular}
A tree $T $
 is  a \textit{regular $\kappa$-tree} if
$T \subseteq {^{<\kappa}\kappa}$ and
\begin{itemize}
\item[(1)]
 $T$ is a $\kappa$-tree; that is,
$T$ has height $\kappa$
 and each level of $T$ has size $< \kappa$;
\item[(2)]
Every maximal branch of $T$
 has length $\kappa$;
\item[(3)]
 $T$ is perfect, meaning that
 for any $t \in T$, there are incomparable
 $s,u \sqsupseteq t$ in $T$.
\end{itemize}
\end{definition}

Note that if $\kappa$ is a regular cardinal
 and there exists a regular $\kappa$-tree,
 then $\kappa$ must be a strongly inaccessible
 cardinal.
However, there do exist regular $\kappa$-trees
 for singular cardinals $\kappa$.
Specifically,
 there exists a regular $\kappa$-tree if
 $2^\lambda < \kappa$ for all
 $\lambda < \cf(\kappa)$.

Given a set $X \subseteq {^{<\kappa}\kappa}$
 and an ordinal $\alpha < \kappa$, let
 $X(\alpha)$ denote the set of sequences in $X$ of length $\alpha$; that is,
\begin{equation}
X(\alpha)=\{ t \in X : \Dom(t) = \alpha \}.
\end{equation}
Given sets $X, Y \subseteq {^{<\kappa}\kappa}$,
 we say that $X$ \textit{dominates} $Y$ if
to each $ y \in Y$ there corresponds at least one  $x \in X$ such that
$ x \sqsupseteq y$.
Given $t \in {^{<\kappa}\kappa}$,
 $\mbox{Cone}(t)$ is the set of all
 $t' \sqsupseteq t$ in ${^{<\kappa}\kappa}$.

\begin{definition}
Given $1 \le d < \omega$
 and a sequence $\langle X_i \subseteq
 {^{<\kappa}\kappa} : i < d \rangle$,
 define the \textit{level product}
 of the $X_i$'s to be
 $$\bigotimes_{i < d} X_i :=
 \{ \langle x_i : i < d \rangle :
 (\exists \alpha < \kappa)
 (\forall i < d)\,
 x_i \in X_i(\alpha) \}.$$
\end{definition}

\begin{definition}
Let $1 \le d < \omega$.
Given $\kappa$-trees $T_0, ..., T_{d-1}$,
 we call a sequence
 $\langle X_i : i < d \rangle$
 a \textit{somewhere dense level matrix} if
 there are ordinals $\alpha < \beta < \kappa$
 and a sequence $\langle t_i \in T_i(\alpha) : i < d \rangle$
 such that each $X_i$ is a subset of $ T_i(\beta)$, and further, $X_i$ dominates
 $T_i(\alpha + 1) \cap \mbox{Cone}(t_i)$.
\end{definition}

The following is the
 \textit{somewhere dense} version of the
 Halpern-L\"auchli Theorem, which we  denote by
 $\textrm{SDHL}(d,\sigma,\kappa)$.
Given a coloring $c$ and a set $S$,
 we say $c$ is monochromatic on $S$
 if and only if $|c``S| = 1$.

\begin{definition}
For $1 \le d < \omega$ and
 cardinals $0 < \sigma < \kappa$
 with $\kappa$ infinite,
 $\mbox{SDHL}(d,\sigma,\kappa)$
 is the statement that given
 any sequence
 $\langle T_i : i < d \rangle$
 of regular  $\kappa$-trees and any coloring
 $$c : \bigotimes_{i < d} T_i \to \sigma,$$
 %there exits $\zeta < \zeta' < \kappa$,
 %$\langle t_i \in T_i(\zeta) : i < d \rangle$, and
 %$\langle X_i \subseteq T_i(\zeta') : i < d \rangle$ such that
 %each $X_i$ dominates $T_i(\zeta + 1) \cap \mbox{Cone}(t_i)$
 %and
 there exists a somewhere dense level matrix
 $\langle X_i \subseteq T_i : i < d \rangle$,
 such that  $c$
 is monochromatic on $\bigotimes_{i < d} X_i$.
 %, meaning
 %$$|c``\bigotimes_{i < d} X_i | = 1.$$
\end{definition}

When we say $\mbox{SDHL}(d,\sigma,\kappa)$ is true,
 this implies that $\mbox{SDHL}(d,\sigma,\kappa)$ is defined,
 so $1 \le d < \omega$ and $\sigma < \kappa$.
However, $\mbox{SDHL}(d,\sigma,\kappa)$ does not imply
 $\kappa$ is an inaccessible cardinal.
If $1 \le d < \omega$ and $0 < \sigma < \kappa$ but
 there are no regular $\kappa$-trees,
 then $\mbox{SDHL}(d,\sigma,\kappa)$ is vacuously true.
This convention makes Proposition~\ref{prop5.1} more convenient.

Often one wants to apply $\mbox{SDHL}$,
 but restricted to certain levels.
This is readily seen to be possible
 for regular cardinals.

\begin{fact}
\label{sdhl_they_can_be_strong}
Let $1 \le d < \omega$,
 $\kappa$ a regular cardinal,
and  $0 < \sigma < \kappa$ be given, and
assume $\mbox{SDHL}(d,\sigma,\kappa)$ holds.
Let $T_i$  ($i < d$) be a sequence
 of regular $\kappa$-trees
 (so $\kappa$ is strongly inaccessible).
Let $c : \bigotimes_{i < d} T_i \to \sigma$ be a coloring, and
let $A \subseteq \kappa$ be cofinal in $\kappa$.
Then
 there exist $\alpha < \beta < \kappa$ both in $A$,
 $\langle t_i \in T_i(\alpha) : i < d \rangle$
 and $\langle X_i \subseteq T_i(\beta) : i < d \rangle$
 such that each $X_i$ dominates
 $T_i(\alpha+1) \cap \mbox{Cone}(t_i)$
 and $c$ is monochromatic on
 $\bigotimes_{i < d} X_i$.
\end{fact}

The other two forms of the Halpern-L\"auchli Theorem
 we will consider involve the notion of a strong subtree.
In this paper, by successor we mean \textit{immediate} successor.

\begin{definition}
Let $T$ be a
 regular $\kappa$-tree.
A tree $T' \subseteq T$ is a
 \textit{strong subtree} of $T$
 as witnessed by some set $A \subseteq \kappa$
 cofinal in $\kappa$ if $T'$ is regular and
 for each $t \in T'(\alpha)$ for $\alpha < \kappa$,
\begin{itemize}
\item[1)] $\alpha \in A$ implies every successor
 of $t$ in $T$ is also in $T'$;
\item[2)] $\alpha \not\in A$ implies that $t$
 has a unique successor in $T'$ on level
 $\alpha + 1$.
\end{itemize}
We  refer to an ordinal  $\alpha \in A$
 as a \textit{splitting level} of $T'$.
\end{definition}

The following is the strong tree version of the
 Halpern-L\"auchli Theorem, which we  denote by
 $\textrm{HL}(d,\sigma,\kappa)$.

\begin{definition}
For $1 \le d < \omega$
 and cardinals $0 < \sigma < \kappa$ with $\kappa$ infinite,
 $\textrm{HL}(d,\sigma,\kappa)$
 is the following statement:
Given any sequence
 $\langle T_i  :
 i < d \rangle$
 of regular $\kappa$-trees and a coloring
 $c:\bigotimes_{i<d} T_i \rightarrow \sigma$,
 there exists a sequence of trees
 $\langle T'_i : i < d \rangle$
 such that the following hold:
\begin{enumerate}
\item
 Each $T'_i$ is a strong subtree of $T_i$
 as witnessed by the same set $A \subseteq \kappa$,
 independent of $i$; and
\item
$c$ is monochromatic on
$\bigcup_{\alpha \in A} \bigotimes_{i<d} T'_i(\alpha)$.
\end{enumerate}
\end{definition}

Just as in Fact~\ref{sdhl_they_can_be_strong},
 HL can be applied but restricted to any $A \subseteq \kappa$
 cofinal in $\kappa$.
Specifically, if $\kappa$ is a regular cardinal,
 and we have a sequence of regular $\kappa$-trees $T_i$ for $i < d$,
 then given any $A \subseteq \kappa$ cofinal in $\kappa$,
 there is a sequence of strong subtrees
 $T'_i \subseteq T_i$ for $i < d$ all witnessed by
 the same set of splitting levels $B \subseteq A$ and
 $c$ is monochromatic on
 $\bigcup_{\beta \in B} \bigotimes_{i < d} T'_i(\beta)$.
In practice, one usually uses  regular $\kappa$-trees $T_i$
such that for each $\alpha\in A$, each node  of length $\alpha$ in $T_i$ has
 two or more successors.

Although in
 \cite{Dobrinen/Hathaway17} we stated that
 $\textrm{HL}(d,\sigma,\kappa)$
 is equivalent to $\textrm{SDHL}(d,\sigma,\kappa)$ for any weakly compact $\kappa$,
 our proof never actually used the assumption
 that $\kappa$ was weakly compact, only that is
 was strongly inaccessible.
% Although in
%  \cite{Dobrinen/Hathaway17} we stated that $\textrm{HL}(d,\sigma,\kappa)$
%  is equivalent to $\textrm{SDHL}(d,\sigma,\kappa)$ for any weakly compact $\kappa$,
% in fact, our work proved
%  that they are equivalent for any inaccessible $\kappa$;
%  this more general  fact
%  appears in our addendum  \cite{Dobrinen/Hathaway17Addendum}:

\begin{proposition}\label{prop.HL=SDHL}
$\textrm{HL}(d,\sigma,\kappa)$
 is equivalent to $\textrm{SDHL}(d,\sigma,\kappa)$ for all inaccessible $\kappa$.
\end{proposition}

Finally, the following is the \textit{tail cone}
 version of the
 Halpern-L\"auchli Theorem,
 which we shall denote by
 $\mbox{HL}^{tc}(d,<\kappa,\kappa)$, which appears in Section 3 of \cite{Zhang17}.

\begin{definition}\label{defn.tailcone}
For $1 \le d < \omega$ and
 $\kappa$ an infinite cardinal,
  $\mbox{HL}^{tc}(d,<\kappa,\kappa)$
  is the following statement:
 Given a sequence of regular $\kappa$-trees
 $\langle T_i  : i < d \rangle$,
 a sequence of nonzero cardinals
 $\langle \sigma_{\zeta} < \kappa : \zeta < \kappa \rangle$,
 and a sequence of colorings
 $c_{\zeta} : \bigotimes_{i < d} T_i \to \sigma_{\zeta}$
 for ${\zeta} < \kappa$,
 there exists a sequence of trees
 $\langle T'_i : i < d \rangle$ such that
\begin{enumerate}
\item[(1)]
Each $T'_i$ is a strong subtree of $T_i$
 as witnessed  by the same set $A \subseteq \kappa$,
 independent of $i$;
  and
\end{enumerate}
letting
$\{\alpha_{\zeta} : {\zeta} < \kappa \}$ denote the increasing enumeration
 of $A$,
\begin{enumerate}
\item[(2)]
For each
 pair of ordinals ${\zeta} \le \xi < \kappa$,
 given any sequence
 $\langle t_i \in T'_i(\alpha_{\xi}) : i < d \rangle$,
 we have
 $$c_{\zeta} \langle t_i : i < d \rangle  =
  c_{\zeta} \langle t_i \restriction \alpha_{\zeta} : i < d \rangle .$$
\end{enumerate}
In other words, the $c_{\zeta}$-color of a tuple
 $\vec{t} = \langle t_i : i < d \rangle$ on the
 $\xi$-th splitting level
 (for $\xi\ge {\zeta}$) is the same as
 the $c_{\zeta}$-color of $\vec{t}$ restricted
 to the $\zeta$-th splitting level.
\end{definition}

Note that
 $\mbox{SDHL}(d,\sigma,\kappa)$,
 $\mbox{HL}(d,\sigma,\kappa)$, and
 $\mbox{HL}^{tc}(d,\sigma,\kappa)$
 are all statements about $V_{\kappa+1}$.
We will need the following
 concept later.

\begin{definition}\label{tcmod}
For $1 \le d < \omega$ and $\kappa$
 an infinite cardinal,
 the \textit{modified}
 $\mbox{HL}^{tc}(d, {<\kappa}, \kappa)$
 is just $\mbox{HL}^{tc}(d, {<\kappa}, \kappa)$
 but with (2) replaced with the following:
\begin{itemize}
\item[(2*)] There is a function
 $\mu : \kappa \to \kappa$ such  that
 $(\forall \gamma < \kappa)\, \mu(\gamma) \ge \gamma$
satisfying  the following:
 For each pair of ordinals
 $ \zeta \le \gamma < \kappa$,
 given any sequence
 $\langle t_i \in T_i'(\alpha_\gamma) : i < d \rangle$,
 we have
 $$c_{\mu(\zeta)} \langle t_i : i < d \rangle =
 c_{\mu(\zeta)} \langle t_i \restriction
 \alpha_{\zeta} : i < d \rangle.$$
In other words, the $c_{\mu(\zeta)}$-color of a tuple
 is determined by restricting to the
 $\zeta$-th splitting level.
\end{itemize}
\end{definition}

\begin{proposition}\label{prop.2.10}
\label{kappamanyisokay}
Fix $1 \le d < \omega$ and $\kappa$
 a strongly inaccessible cardinal.
Then $\mbox{HL}^{tc}(d,{<\kappa},\kappa)$
 and its modified version are equivalent.
\end{proposition}

\begin{proof}
It is clear that the unmodified version
 implies the modified version holds:
Just set
 $\mu : \kappa \to \kappa$ to be
 the identity function.
For the other direction,
 assume the modified version holds.

Let $\langle T_i : i < d \rangle$
 be a sequence of regular $\kappa$-trees and
 $\langle c_{\zeta} : \zeta < \kappa \rangle$
 a sequence of colorings, where
 $c_{\zeta} : \bigotimes_{i<d} T_i \to \sigma_{\zeta}$
 for each $\zeta < \kappa$.
We will find strong subtrees
 $T_i' \subseteq T_i$ for $i < d$ all witnessed by
 the same set of splitting levels
 $A = \{ \alpha_{\gamma} : \gamma < \kappa \}$
 such that (2) of the definition of
 $\mbox{HL}^{tc}$ holds.

For each $\gamma < \kappa$,
 let $\sigma'_{\gamma}$ be the product of the cardinals
 $\sigma_{\zeta}$ for $\zeta \le \gamma$.
Since $\kappa$ is strongly inaccessible,
 each $\sigma'_{\gamma}$ is strictly less than $\kappa$.
For each $\gamma < \kappa$,
 let $c'_{\gamma} : \bigotimes_{i < d} T_i \to \sigma'_{\gamma}$
 be a coloring which encodes the colorings
 $c_{\zeta}$ for $\zeta \le \gamma$.
That is, given $\gamma < \kappa$ and
 $\vec{t} \in \bigotimes_{i < d} T_i$,
 the sequence
 $\langle c_{\zeta}(\vec{t}\,) : \zeta \le \gamma \rangle$
 can be recovered from knowing $c_{\gamma}'(\vec{t}\,)$.
Thus,
given any
 $\vec{s}, \vec{t}$
 with the same $c_{\gamma}'$-color,
 then $\vec{s}, \vec{t}$ have the same
 $c_{\zeta}$-color for all $\zeta \le \gamma$.

Now apply the modified version
 to the trees $\langle T_i : i < d \rangle$
 and the colorings
 $\langle c'_{\gamma} : \gamma < \kappa \rangle$
 to produce strong subtrees
 $T_i' \subseteq T_i$,
 each with the same set of splitting levels
 $A = \{ \alpha_{\gamma} : \gamma < \kappa \}$,
 and some  fixed function
 $\mu: \kappa \to \kappa$ such that
 for any pair of ordinals $\zeta \le \gamma < \kappa$,
 the $c'_{\mu(\zeta)}$-color of a $d$-tuple
 on the splitting level $\alpha_{\gamma}$
 is determined by restricting to
 the splitting level $\alpha_{\zeta}$.
Then
 for any pair of ordinals $\zeta \le \gamma < \kappa$
 and any $\langle t_i \in T_i'(\alpha_\gamma) :
 i < d \rangle$,
\begin{equation}
 c'_{\mu(\zeta)} \langle t_i : i < d \rangle =
 c'_{\mu(\zeta)} \langle t_i \restriction \alpha_{\zeta} :
 i < d \rangle.
\end{equation}
Recalling that  $c'_{\mu(\zeta)}$
 encodes the colorings $c_{\psi}$ for all $\psi \le \mu(\zeta)$
and that the function $\mu$  satisfies $\mu(\gamma)\ge\gamma$ for each $\gamma<\kappa$,
it follows that
 for each pair of ordinals
 $\zeta \le \gamma < \kappa$
and
 any $\langle t_i \in T_i'(\alpha_{\gamma}) :
 i < d \rangle$,
 \begin{equation}
c_{\zeta} \langle t_i : i < d \rangle =
 c_{\zeta} \langle t_i \restriction \alpha_{\zeta} :
 i < d \rangle.
\end{equation}
This is precisely (2) of the definition
 of $\mbox{HL}^{tc}$.
\end{proof}

\begin{observation}
\label{obs1}
$\mbox{HL}^{tc}(d,{<\kappa},\kappa)$
 implies $(\forall \sigma < \kappa)\,
 \mbox{HL}(d,\sigma,\kappa)$.
This can be seen by using a sequence of colorings
 $\langle c_{\zeta} : \zeta < \kappa \rangle$
 that are all constant except the first one
 $c_0$.
Applying $\mbox{HL}^{tc}(d,{<\kappa},\kappa)$ produces
  strong subtrees $T_i' \subseteq T_i$
 such that the $c_0$-color of a tuple on
 any splitting level is determined by restricting to
 the $0$-th splitting level.
\end{observation}

\section{Derived Trees}\label{section.Derived}

This section introduces derived trees and proves a
 theorem which will be central to the results in
 Section \ref{sec.4} about small forcings preserving
 various forms of the Halpern-\Lauchli\ Theorem.

\begin{definition}\label{defn.DerT}
Let $\kappa$ be a cardinal,
 $\mbb{P}$ be  a forcing, and
without loss of generality, assume that $\mbb{P}$ has a largest member, denoted $1$.
Assume that  $\dot{T}$ is  a $\mbb{P}$-name  for which
$1$
forces that
$ \dot{T}$ is a subtree of
${^{<\check{\kappa}} \check{\kappa}}$.
The \defemph{derived tree of} $\dot{T}$,
denoted $\mbox{Der}(\dot{T})$,
 is defined as follows.
The elements of $\mbox{Der}(\dot{T})$
 are equivalence classes of pairs
 $(\dot{\tau}, \alpha)$ satisfying
 \begin{equation}\label{eq.4}
1 \forces ( \dot{\tau} \in \dot{T} \mbox{ and }
 \mbox{Dom}(\dot{\tau}) = \check{\alpha}),
\end{equation}
 where the equivalence relation $\cong$ is
 defined by
 \begin{equation}\label{eq.5}
(\dot{\tau}_1, \alpha_1) \cong
  (\dot{\tau}_2, \alpha_2)
 \Longleftrightarrow
1 \forces (\dot{\tau}_1 = \dot{\tau}_2).
\end{equation}
Notice that if $(\dot{\tau}_1, \alpha_1) \cong
 (\dot{\tau}_2, \alpha_2)$, then $\alpha_1 = \alpha_2$.
The elements of $\mbox{Der}(\dot{T})$
 are ordered as follows:
 \begin{equation}
[(\dot{\tau}_1, \alpha_1)] <
  [(\dot{\tau}_2, \alpha_2)]
 \Longleftrightarrow
1 \forces
 (\dot{\tau}_1 \sqsubseteq \dot{\tau}_2
 \mbox{ and } \dot{\tau}_1 \not= \dot{\tau}_2).
\end{equation}
Given $S \subseteq \mbox{Der}(\dot{T})$, let
\begin{equation}
\mbox{Names}(S)=\{\dot{\tau}:(\exists \alpha)\,
 [(\dot{\tau},\alpha)] \in S\}.
\end{equation}
\end{definition}

%Notice that the relation defined in equation (\ref{eq.5})
%is truly an equivalence relation, since $1\forces
%(\dot{\tau}_1=\dot{\tau}_2)$ implies that
%$\alpha_1=\alpha_2$ by (\ref{eq.4}).

We claim  that $1$ forces that
 every element of $\dot{T}$ is  equal to
 some element of $\mbox{Names}(\mbox{Der}(\dot{T}))$.
To see why, let $G$ be $\mbb{P}$-generic over $V$.
Let $t \in \dot{T}_G$ and
 $\alpha = \mbox{Dom}(t)$.
Fix a name  $\dot{\tau}$ such that
 $\dot{\tau}_G = t$, and
let $\dot{b}$ be a name for
 the leftmost branch of $\dot{T}$.
Define $\dot{\rho}$  so  that
\begin{equation}
 1 \forces [( \dot{\rho} = \dot{\tau}\mathrm{\ if\ }
 \mbox{Dom}(\dot{\tau}) = \check{\alpha})\wedge(
 \dot{\rho} = \dot{b} \restriction \alpha \mathrm{\ if\ }
 \mbox{Dom}(\dot{\tau}) \not= \check{\alpha})].
\end{equation}
Then  $[(\dot{\rho},\alpha)] \in \mbox{Der}(\dot{T})$ and
 $\dot{\rho}_G = t$.

We will now show that $\mbox{Der}(\dot{T})$
 is  (isomorphic to) a regular $\kappa$-tree in the ground model
 whenever $1$ forces that  $\dot{T}$ is a regular $\kappa$-tree,
 and that given an element named
 by some $(\dot{\tau},\alpha)
 \in \mbox{Der}(\dot{T})$,
 all  its successors
 are named by
 successors of
 $(\dot{\tau},\alpha)$ in
$\mbox{Der}(\dot{T})$.
This theorem is central to the forcing preservation theorems in following sections.

% Given a tree $T \subseteq {^{<\kappa}\kappa}$
%  and a node $t \in T$,
%  by the \textit{first} successor of $t$
%  we mean the node $t ^\frown \alpha \in T$
%  with the least possible $\alpha \in \kappa$.
% More generally,
%  the $\gamma$-th successor of $t \in T$
%  is the node $t ^\frown \alpha \in T$
%  such that the set
%  $\{ \beta < \alpha : t ^\frown \beta \in T \}$
%  has order type $\gamma$.
Given a tree $T \subseteq {^{<\kappa}\kappa}$
 and a node $t \in T$,
 by the $0$-th successor of $t$
 we mean the node $t ^\frown \alpha \in T$
 with the least possible $\alpha \in \kappa$.
More generally,
 the $\gamma$-th successor of $t \in T$
 is the node $t ^\frown \alpha \in T$
 such that the set
 $\{ \beta < \alpha : t ^\frown \beta \in T \}$
 has order type $\gamma$.

\begin{theorem}[Derived Tree Theorem]\label{mainlemma}
Let $\kappa$ be strongly inaccessible,
 $\mbb{P}$ a forcing of size $< \kappa$, and
 $\dot{T}$ a name for a regular $\kappa$-tree.
Then $\mbox{Der}(\dot{T})$ is isomorphic to a regular $\kappa$-tree and
\begin{enumerate}
\item[$(*)$]
 If $[(\dot{\tau}, \alpha)] \in
 \mbox{Der}(\dot{T})$ and $X$ is the
 set of all $\dot{\rho}$ such that
 $[(\dot{\rho}, \alpha+1)]$ is a successor
 of $[(\dot{\tau},\alpha)]$
 in $\mbox{Der}(\dot{T})$, then
 $1 \forces ($every successor of
 $\dot{\tau}$ in $\dot{T}$
 is named by an element of $\check{X})$.
\end{enumerate}
\end{theorem}

\begin{proof}
First note that if
 $[(\dot{\tau},\alpha)]$ is in
 $\mbox{Der}(\dot{T})$ and
 $\beta < \alpha$,
 then there is a name
 $\dot{\tau}_\beta$ such that
 $[(\dot{\tau}_\beta,\beta)]$
 is in $\mbox{Der}(\dot{T})$
 and  $[(\dot{\tau}_\beta,\beta)]< [(\dot{\tau},\alpha)]$:
 simply let $\dot{\tau}_\beta$ be a name
 for $\dot{\tau} \restriction \check{\beta}$.
We prove that $\mbox{Der}(\dot{T})$ is a regular
$\kappa$-tree by proving it satisfies conditions (1) - (3) of Definition
\ref{defn.regular}.

To verify (1), we must first show that
 $\mbox{Der}(\dot{T})$ is a tree.
Suppose $[(\dot{\tau}_1, \alpha_1)],
 [(\dot{\tau}_2, \alpha_2)],
 [(\dot{\tau}_3, \alpha_3)]$
 are members of  $\mbox{Der}(\dot{T})$
satisfying
\begin{equation}
[(\dot{\tau}_1, \alpha_1)]>
 [(\dot{\tau}_2, \alpha_2)] \mathrm{\ and\ }[(\dot{\tau}_1, \alpha_1)]>
 [(\dot{\tau}_3, \alpha_3)].
\end{equation}
Assume, without loss of generality,
 that
 $\alpha_2 \ge \alpha_3$.
Since  $1$  forces that $\dot{T}$ is a tree and that         $ \dot{\tau}_i$ is an initial segment of $\dot{\tau}_1$ of length $\alpha_i$, for $i\in \{2,3\}$,
it follows that $1$ forces that $\dot{\tau}_2\restriction \alpha_3=\dot{\tau}_3$.
Thus,
 $1$ forces that
$\dot{\tau}_3$ is an initial segment of
 $\dot{\tau}_2$,
and hence,
 $[(\dot{\tau_2}, \alpha_2)]$ and
 $[(\dot{\tau_3}, \alpha_2)]$
 are comparable in
 $\mbox{Der}(\dot{T})$.

For $\beta<\kappa$,
let {\em level $\beta$} denote the set of those $[(\dot{\tau},\alpha)]\in\mbox{Der}(\dot{T})$ such that $\alpha=\beta$.
The same argument as above also shows that given
 $[(\dot{\tau},\alpha)]$
 in $\mbox{Der}(\dot{T})$
 and $\beta < \alpha$,
 there is a \textit{unique}
 $[(\dot{\rho}, \beta)]$
 on level $\beta$ of
 $\mbox{Der}(\dot{T})$
 such that  $[(\dot{\rho}, \beta)]< [(\dot{\tau},\alpha)]$.
We have now established that
 $\mbox{Der}(\dot{T})$ is a tree.

We must now show that
 $\mbox{Der}(\dot{T})$
 is a $\kappa$-tree.
To show that it has height $\kappa$,
 given any $\alpha < \kappa$,
 let $\dot{\tau}_\alpha$ be such that
 $1 \forces (\dot{\tau}_\alpha =
 \dot{b} \restriction \check{\alpha})$,
 where $\dot{b}$ is a name for the leftmost branch of $\dot{T}$.
Then $[(\dot{\tau}_\alpha,\alpha)]
 \in \mbox{Der}(\dot{T})$.
Thus, $\mbox{Der}(\dot{T})$ has height $\kappa$.
To
 show that each level of
 $\mbox{Der}(\dot{T})$ has
 $< \kappa$ nodes, we will make
 use of the fact
 that $\mbox{Der}(\dot{T})$ consists
 of elements $[(\dot{\tau},\alpha)]$
 where $1 \forces (\mbox{Dom}(\dot{\tau})
 = \check{\alpha})$.
(If we drop the $\alpha$'s from
 the definition of
 $\mbox{Der}(\dot{T})$, we can verify $(*)$, and
(2) and  (3) of  Definition
\ref{defn.regular},  but not (1).)
Since $1 \forces (\dot{T}$ is a
 $\check{\kappa}$-tree$)$,
 we have that
 \begin{equation}\label{eq.10}
1 \forces
 (\forall \alpha < \check{\kappa})(\exists \gamma < \check{\kappa})
 (\forall t \in \dot{T})\,
 \alpha \in \mbox{Dom}(t)
 \Rightarrow
 t(\alpha) < \gamma.
\end{equation}
Since  $|\mbb{P}|<\kappa$,
 there is a function $g : \kappa \to \kappa$
 such that
 \begin{equation}\label{eq.11}
1 \forces (\forall \alpha < \check{\kappa})(\forall t \in \dot{T})\,
 \alpha \in \mbox{Dom}(t)
 \Rightarrow
 t(\alpha) <
 \check{g}(\alpha).
\end{equation}

Now, to each  pair
 $(\dot{\tau}, \alpha)$, where
 $[(\dot{\tau}, \alpha)] \in
 \mbox{Der}(\dot{T})$, we
 may associate   a sequence
 $\langle f_\xi : \xi < \alpha \rangle$,
 where each $f_\xi$ is a function
 from some maximal antichain of
 $\mbb{P}$ to $g(\xi)$.
This sequence represents a nice name for $\dot{\tau}$.
Since $|\mbb{P}|<\kappa$,
 level $\alpha < \kappa$ of
 $\mbox{Der}(\dot{T})$ is bounded from above
 by the following:
\begin{equation}\label{eq.12}
\prod_{\xi < \alpha}
 g(\xi)^{|\mbb{P}|}.
\end{equation}
Since $\kappa$ is strongly inaccessible,
 this bound is $<\kappa$.
We have now shown (1) of  Definition
\ref{defn.regular}.

We will now verify (2),
 that $\mbox{Det}(\dot{T})$ has no maximal
 branches of length $< \kappa$.
When we later show that $\mbox{Der}(\dot{T})$ is perfect,
 this will imply it has no maximal branches
 of a successor ordinal length.
Thus, it suffices to show $\mbox{Der}(\dot{T})$
 has no maximal branches of limit length.
Let $\eta < \kappa$ be a limit ordinal and
 $S = \langle [(\dot{\tau}_\alpha,\alpha)] :
 \alpha < \eta \rangle$ is an increasing chain in $\mbox{Der}(\dot{T})$ so that  for all $\xi<\zeta<\eta$,
\begin{equation}
[(\dot{\tau}_{\xi}, \alpha_{\xi})] <
  [(\dot{\tau}_{\zeta}, \alpha_{\zeta})].
\end{equation}
Let $\dot{s}$ be a name  which $\mbb{P}$
 forces to be a function from  $\check{\eta}$ to
 ${^{<\check{\kappa}}\check{\kappa}}$ such that
 for all $\alpha < \eta$,
\begin{equation}\label{eq.13}
 1 \forces \dot{s}(\check{\alpha})
 = \dot{\tau}_\alpha.
\end{equation}
By the definition of the tree relation $<$ in $\mbox{Der}(\dot{T})$,
it follows that
\begin{equation}
\label{eq_2}
1 \forces (\forall \alpha < \beta < \check{\eta})\,
 \dot{s}(\alpha) \sqsubseteq \dot{s}(\beta).
\end{equation}
Now let $\dot{\tau}_\eta$ be a name such that
 $1 \forces \dot{\tau}_\eta =
 \bigcup_{\alpha < \check{\eta}} \dot{s}(\alpha)$.
It follows from  (\ref{eq_2}) that
\begin{equation}
 1 \forces \dot{\tau}_\eta \in
 {^{\check{\eta}} \check{\kappa}}\mathrm{\ and\ }
 (\forall \alpha < \check{\eta})\,
 \dot{s}(\alpha) \sqsubseteq \dot{\tau}_\eta,
\end{equation}
and thus, by (\ref{eq.13}),
\begin{equation}
1 \forces (\forall \alpha < \check{\eta})\,
 \dot{\tau}_\alpha \sqsubseteq \dot{\tau}_\eta.
\end{equation}
Since $\mbb{P}$  forces that $\dot{T}$ has no maximal branches
 of length $<\kappa$,  we now have that
 \begin{equation}
1 \forces \dot{\tau}_\eta \in \dot{T}.
\end{equation}
So now $[(\dot{\tau}_\eta,\eta)] \in
 \mbox{Der}(\dot{T})$,
 and this node is above each
 $[(\dot{\tau}_\alpha, \alpha)]$
 for $\alpha < \eta$.
Thus, we have verified (2) of Definition
\ref{defn.regular}.

To verify (3),
 consider any $[(\dot{\tau},\alpha)] \in
 \mbox{Der}(\dot{T})$.
Let $\dot{b}$ be a name for
 the leftmost
 branch of $\dot{T}$ which extends
 $\dot{\tau}$.
Let $\beta<\kappa$ be the least ordinal  greater than or equal to  $\alpha$
 for which  there is some $p \in \mbb{P}$
which forces that
there are at least two
 successors of $\dot{b} \restriction
 \check{\beta}$ in the tree $\dot{T}$.
Such $\beta$ and $p$ exist since $\mbb{P}$ forces that $\dot{T}$ is a perfect tree.
Let $\dot{\tau}_1$ be a name such that
 $1$ forces  $\dot{\tau}_1 = \dot{b}
 \restriction \check{\beta}$.
Let $\dot{\tau}_2$ be a name  which
 % $1$ forces to be first successor of $\dot{\tau}_1$
 $1$ forces to be the $0$-th successor of $\dot{\tau}_1$
 in $\dot{T}$.
Finally, let  $\dot{\tau}_3$ be a name which
 $1$ forces  to be  the
 % second successor of $\dot{\tau}_1$
 $1$-th successor of $\dot{\tau}_1$
 in $\dot{T}$,
 if there are at least two successors,
 and the unique successor if there is only one successor.
One can see that
 $[(\dot{\tau}_2, \beta + 1)]$ and
 $[(\dot{\tau}_3, \beta + 1)]$
 are successors of
 $[(\dot{\tau}_1, \beta)]$ in
 $\mbox{Der}(\dot{T})$.
Since there is some $p$  which
 forces $\dot{\tau}_2 \not=
 \dot{\tau}_3$,
it follows that
 $[(\dot{\tau}_2, \beta + 1)] \not=
  [(\dot{\tau}_3, \beta + 1)]$.
Thus,
 $[(\dot{\tau}_2, \beta + 1)]$ and
 $[(\dot{\tau}_3, \beta + 1)]$
 are incomparable extensions of
 $[(\dot{\tau}, \alpha)]$
 in $\mbox{Der}(\dot{T})$.
Therefore, $\mbox{Der}(\dot{T})$ is a perfect tree.
 Hence,  $\mbox{Der}(\dot{T})$ is isomorphic to  a regular $\kappa$-tree.

Finally, the verification of ($*$)
 follows almost immediately from
 the definition of
 $\mbox{Der}(\dot{T})$.
Fix
 $[(\dot{\tau},\alpha)] \in \mbox{Der}(\dot{T})$ and
let $G$ be $\mbb{P}$-generic over $V$.
Let $s$ be an arbitrary successor of
 $\dot{\tau}_G$ in $\dot{T}_G$, and
let $\gamma$ be such that
 $s$ is the $\gamma$-th successor of
 $\dot{\tau}_G$ in $\dot{T}_G$.
Take  $\dot{\rho}$ to be a name  so that
 $1$ forces that $\dot{\rho}$ is the $\check{\gamma}$-th
 successor of $\dot{\tau}$ in $\dot{T}$,
 if it exists,
and the $0$-th successor, otherwise.
Then
 $\dot{\rho}_G = s$.
At the same time
 $1$ forces that $\dot{\rho}$
 is a successor of $\dot{\tau}$
 in $\dot{T}$,
 so
 $[(\dot{\rho},\alpha+1)]$ is a member of
$ \mbox{Der}(\dot{T})$.
\end{proof}

\begin{remark}\label{rem.kappaccnotenough}
There are two instances in the proof where $|\mathbb{P}|<\kappa$ was used.
The first is non-essential:
If $\mathbb{P}$ is $\kappa$-c.c., or even just    $(\kappa, \kappa, <\kappa)$-distributive,
equation (\ref{eq.11}) still  holds.
However, the  second use of $|\mathbb{P}|<\kappa$ 
is essential to the proof.
Indeed, given any  $\mathbb{P}$  which  preserves $\kappa$  and has cardinality at least $\kappa$, 
 there is  a name $\dot{T}$ for a regular $\kappa$-tree with the following properties:
 $1$ forces that the first
 level of $\dot{T}$ has size at least two
 (with say elements $\langle 0 \rangle$ and $\langle 1 \rangle$),
 and 
  letting $\{p_{\zeta}:\zeta<\kappa\}$  be a set of distinct members of $\mathbb{P}$,
there are nice names $\dot{\tau}_{\zeta}$
so that
\begin{equation}
p_{\zeta}\forces
\dot{\tau}_{\zeta}\in\dot{T},\
\Dom(\dot{\tau}_{\zeta})=1, \mathrm{\ and \ }\dot{\tau}_{\zeta}(0)=0,
\end{equation}
 and all  $q\in\mathbb{P}$ incompatible with $p_{\zeta}$ force $\dot{\tau}_{\zeta}(0)=1$.
Then for all $\zeta<\xi<\kappa$,
$(\dot{\tau}_{\zeta},1)\not\cong (\dot{\tau}_{\xi},1)$,
so  the first level of Der$(\dot{T})$ has size at least $\kappa$.
Thus,  the  $\kappa$-c.c.\ is not enough to guarantee that the levels of  Der$(\dot{T})$ have size less than $\kappa$.
\end{remark}

The Derived Tree Theorem will allow us to
use $\mbox{Der}(\dot{T})$
 in theorems that require subtrees of
 ${^{<\kappa}\kappa}$.
Note also  that the Derived Tree Theorem provides the following:
Given a name for a regular $\kappa$-tree $\dot{T}$,
 a strong subtree $S$ of
 $\mbox{Der}(\dot{T})$
 with splitting levels $A \subseteq \kappa$,
 and a $\mbb{P}$-generic $G$,
 the set
 $W = \{ \dot{\tau}_G :
 \dot{\tau} \in \mbox{Names}(S) \}$
 is a strong subtree of $\dot{T}_G$,
 witnessed by the set of splitting
 levels $A$.

%%%%%%%%%%%%%%%%%
%%%%%%%%%%%%%%%%%
%%%%%%%%%%%%%%%%%

\section{Small Forcings Preserve
 $\mbox{SDHL}$, $\mbox{HL}$, and $\mbox{HL}^{tc}$}\label{sec.4}

In this section we show that if
 $\kappa$ is strongly inaccessible
 and $(\forall \sigma < \kappa)\,
 \mbox{SDHL}(d,\sigma,\kappa)$ holds,
 then this still holds after performing any
 forcing of size less than $\kappa$.
This result then automatically holds  for HL replacing SDHL, since the two are equivalent  for $\kappa$ inaccessible.
Further, we show that
 $\mbox{HL}^{tc}$ at $\kappa$ is preserved by forcings of size less than $\kappa$.
These results make strong use  of
 the Derived Tree Theorem from the previous section.

%We can now show that
% small forcings preserve
% $\mbox{SDHL}(d,\sigma,\kappa)$.
\begin{theorem}\label{thm.4.1}
\label{sdhl_is_preserved}
Let $\kappa$ be strongly inaccessible.
Let $1 \le d < \omega$ and
 $0 < \sigma < \kappa$.
Let $\mbb{P}$ be a forcing of size
 $< \kappa$.
Assume that
 $\mbox{SDHL}(d,\sigma \cdot |\mbb{P}|,\kappa)$
 holds.
Then $\mbox{SDHL}(d,\sigma,\kappa)$ holds
 after forcing with $\mbb{P}$.
In particular,  the statement ``$(\forall \sigma<\kappa)\ \mbox{SDHL}(d,\sigma,\kappa)$ holds" is preserved by all forcings of size less than $\kappa$.
\end{theorem}

\begin{proof}
Let $\langle \dot{T}_i : i < d \rangle$
 be a sequence of names for regular trees
 in the extension.
That is,
 $(\forall i < d)\, 1 \forces (\dot{T}_i \subseteq
 {^{<\check{\kappa}} \check{\kappa}}$
 is regular$)$.
Let $\dot{c}$ be such that
 \begin{equation}
1 \forces \dot{c} :
 \bigotimes_{i < \check{d}}
 \dot{T}_i \to \check{\sigma}.
\end{equation}
We must show that
 $\mbb{P}$ forces that  there is a somewhere dense
 level matrix $\langle X_i \subseteq \dot{T}_i :
 i < \check{d} \rangle$
 such that $|\dot{c}``\bigotimes_{i < \check{d}}
 X_i| = 1$.
We will do this by showing that for each $p\in\mbb{P}$, there is some $p'\le p$ forcing this statement.
Fix $p \in \mbb{P}$.

Consider the trees
 $\mbox{Der}(\dot{T}_i)$ for $i < d$.
Let
\begin{equation}
c' : \bigotimes_{i < d}
 \mbox{Der}(\dot{T}_i) \to \sigma \times \mbb{P}
\end{equation}
 be a coloring defined so  that for any
 $\alpha < \kappa$ and any level $d$-tuple
 \begin{equation}
\vec{t} = \langle [(\dot{\tau}_i, \alpha)]
 \in \mbox{Der}(\dot{T}_i) : i < d \rangle,
\end{equation}
 $c'(\vec{t}\,) = \langle \sigma', q \rangle$
 where $\sigma'$ and $q$ satisfy
 $q \le p$ and
\begin{equation}
 q \forces
 \dot{c}\langle \dot{\tau}_i : i < d \rangle
 = \check{\sigma}'.
\end{equation}
That is, the $c'$-color of $\vec{t}$
 is a pair $\langle \sigma', q \rangle
 \in \sigma \times \mbb{P}$
 where $q$ forces the $\dot{c}$-color of
 the tuple named by $\vec{t}$ to have
 the color $\check{\sigma}'$.

Since $\mbox{SDHL}(d,
 \sigma \cdot |\mbb{P}|, \kappa)$ holds,
 there is a somewhere dense level matrix
 $\langle Y_i \subseteq
 \mbox{Der}(\dot{T}_i) : i < d \rangle$
 that is $c'$-monochromatic.
Let $\xi< \kappa$ be such that each
 $Y_i$ is on level $\xi$ of $\mbox{Der}(\dot{T})$.
Fix $\zeta < \xi$ and
 for each $i < d$,
 fix $[(\dot{\delta}_i, \zeta)]
 \in \mbox{Der}(\dot{T})$ such that
 $Y_i$ dominates the set of successors of
 $[(\dot{\delta}_i, \zeta)]$ in
 $\mbox{Der}(\dot{T}_i)$.
Let $\langle \sigma', p' \rangle$ be the unique
 color assigned to each element of
 $\bigotimes_{i < d} Y_i$ by $c'$.
Hence for all
 $\langle [(\dot{\tau}_i, \xi)] \in Y_i
 : i < d \rangle$,
\begin{equation}
\label{eq_3}
 p' \forces
 \dot{c}\langle \dot{\tau}_i : i < d \rangle
 = \check{\sigma}'.
\end{equation}

We  now show that $p' $ forces that there is a
 $\dot{c}$-monochromatic
 somewhere dense
 level matrix of $\langle \dot{T}_i : i < d \rangle$.
Let $G$ be any $\mbb{P}$-generic  over $V$ containing   $p'$.
It suffices to show that
 in $V[G]$, there is a $\dot{c}_G$-monochromatic
 somewhere dense level matrix of
 $\langle (\dot{T}_i)_G : i < d \rangle.$

For each $i < d$, let
\begin{equation}
X_i = \{ \dot{\tau}_G :
 \dot{\tau} \in \mbox{Names}(Y_i) \}.
\end{equation}
By ($*$) of Theorem~\ref{mainlemma},
 \begin{equation}
(\forall i < d)\,
 X_i \mbox{ dominates the successors of }
 (\dot{\delta}_i)_G \mbox{ in } (\dot{T}_i)_G,
\end{equation}
 so $\langle X_i : i < d \rangle$
 is a somewhere dense level matrix
 of $\langle (T_i)_G : i < d \rangle$.
By (\ref{eq_3})
 and since $p' \in G$,
 we have that
 $$\dot{c}_G``
 \bigotimes_{i < d}
 X_i = \{ \sigma' \},$$
 so $\langle X_i : i < d \rangle$
 is $\dot{c}_G$-monochromatic.
\end{proof}

The Derived Tree Theorem also implies
 that $\mbox{HL}^{tc}$ is preserved by
 small forcings, as we will now see.

\begin{theorem}\label{thm.4.2}
\label{hltail_preserved}
Let $\kappa$ be strongly inaccessible.
Let $1 \le d < \omega$.
Let $\mbb{P}$ be a forcing of size
 $< \kappa$.
Assume that
 $\mbox{HL}^{tc}(d,{<\kappa},\kappa)$
 holds.
Then $\mbox{HL}^{tc}(d,{<\kappa},\kappa)$ holds
 after forcing with $\mbb{P}$.
\end{theorem}

\begin{proof}
The proof is similar to that of the previous theorem.
Fix $p \in \mbb{P}$.
By Proposition~\ref{kappamanyisokay},
 it suffices to find a $p' \le p$ that forces
 the modified version of $\mbox{HL}^{tc}$.
% (where we need only get $\kappa$ many
% $j < \kappa$ to have the required property).
Let $\langle \dot{T}_i : i < d \rangle$
 be a sequence of names for regular trees,
and let $\langle \dot{c}_{\zeta} : \zeta < \kappa \rangle$  be
a sequence of names for colorings
such that $\mbb{P}$ forces each $\dot{c}_{\zeta}$
 to take less than $\check{\kappa}$ colors.
Since $|\mbb{P}| < \kappa$,
 there are ordinals $\sigma_{\zeta}<\kappa$ for $\zeta<\kappa$
 such that
\begin{equation}
 1 \forces \dot{c}_{\zeta} :
 \bigotimes_{i < \check{d}} \dot{T}_i \to
 \check{\sigma}_{\zeta}.
\end{equation}

Just as in the previous theorem,
 the sequence of colorings
 $\langle \dot{c}_{\zeta} : {\zeta} < \kappa \rangle$
 induces a sequence of colorings
 $\langle c'_{\zeta} : \zeta < \kappa \rangle$
 where for each ${\zeta} < \kappa$,
 \begin{equation}
c'_{\zeta} :
 \bigotimes_{i < d} \mbox{Der}(\dot{T}_i)
 \to \sigma_{\zeta} \times \mbb{P}.
\end{equation}
Now apply $\mbox{HL}^{tc}(d,{<\kappa},\kappa)$
 to the sequence of the trees
 $\langle \mbox{Der}(\dot{T}_i) : i < d \rangle$
 and the sequence of colorings
 $\langle c'_{\zeta} : \zeta < \kappa \rangle$.
What results is a sequence of strong subtrees
 $S_i \subseteq \mbox{Der}(\dot{T}_i)$
 for $i < d$, all witnessed by the same set
 of splitting levels $A \subseteq \kappa$.
Let $A$ be enumerated in increasing order as
 $A = \{\alpha_{\zeta} : \zeta < \kappa \}$.
For each pair of ordinals $\zeta \le \xi < \kappa$,
 given any $d$-tuple
 $\vec{t} \in \bigotimes_{i < d} S_i(a_\xi)$,
% on the $\zeta$-th splitting
 %levels of of the $S_i$'s,
 the $c'_{\zeta}$-th color of
 that tuple is the same as the $c'_{\zeta}$-th color of
 that tuple restricted to
 $\bigotimes_{i < d} S_i(\alpha_{\zeta})$.

Recall that for the colorings
 $\langle c'_{\zeta} : {\zeta} < \kappa \rangle$,
 the $c'_{\zeta}$-th color of a tuple is a pair
 $\langle \sigma', q \rangle$.
If the $q$-component of all the tuples
 from the splitting levels of
 the $S_i$ trees are the same,
 then that $q$ forces that the colorings
 are homogenized in the desired way.
In that case, we can set $p' = q$ and be done.
So, the challenge now is to
 further homogenize to make the $q$'s the same.

Let $P : \bigcup_{{\zeta} < \kappa}
 \bigotimes_{i < d}
 S_i(\alpha_{\zeta}) \to \mbb{P}$
 be the following coloring:
 Given ${\zeta} < \kappa$ and
 $\vec{t} = \langle t_i \in S_i(\alpha_{\zeta}) : i < d \rangle$,
 define
\begin{equation}
P(\vec{t}\,) =
 \mbox{ the }q \mbox{-component of }
 c_{\zeta}' \langle t_i : i < d \rangle.
\end{equation}

We mentioned in Observation~\ref{obs1} that
 $\mbox{HL}^{tc}(d, {<\kappa}, \kappa)$ implies
 $\mbox{HL}(d,\sigma,\kappa)$ for all $\sigma < \kappa$.
Apply $\mbox{HL}(d,|\mbb{P}|,\kappa)$ to obtain
  trees $U_i \subseteq S_i$ for $i < d$,
 such that each $U_i$ is a strong subtree of
 $T_i$ as witnessed by a set of splitting
 levels $B \subseteq A$,
 such that for some fixed $p'$,
 \begin{equation}
P`` \bigcup_{\beta \in B}
 \bigotimes_{i<d} U_i(\beta) = \{ p' \}.
 \end{equation}

Now we have that for any
 two ordinals
 $\zeta \le \xi < \kappa$ such that $\alpha_{\zeta} \in B$,
 given any $\vec{t} = \langle [(\dot{\tau}_i, \alpha_\xi)]
 \in U_i(\alpha_\xi) : i < d \rangle$,
 \begin{equation}\label{eq.use}
p' \forces
 \dot{c}_{\zeta} \langle \dot{\tau}_i : i < d \rangle =
 \dot{c}_{\zeta} \langle
 \dot{\tau}_i \restriction \alpha_{\zeta} : i < d \rangle.
\end{equation}

We are now almost done.
The only problem is that each $\alpha_{\zeta}$
 is not necessarily the $\zeta$-th
 splitting level of the $U_i$'s
 (recall that $\alpha_{\zeta}$ is the $\zeta$-th splitting
 level of the $S_i$'s).
Now let $\{ \beta_\gamma : \gamma < \kappa \} = B$
 be the increasing enumeration of $B$.
Let $\mu : \kappa \to \kappa$
 be the function such
 that
\begin{equation}\label{eq.lmu}
\beta_\gamma = \alpha_{\mu (\gamma)}.
\end{equation}
The function $\mu$ is strictly increasing
 and $\gamma \le \mu(\gamma)$ for all $\gamma<\kappa$.
Using (\ref{eq.lmu}) and substituting into equation (\ref{eq.use}),
we obtain that for any
 given $\gamma \le \nu < \kappa$
 and $\vec{t} = \langle [(\dot{\tau}_i, \beta_\nu)]
 \in U_i(\beta_\nu) : i < d \rangle$,
\begin{equation}
\label{mu_eq}
 p' \forces \dot{c}_{\mu(\gamma)}
 \langle \dot{\tau}_i : i < d \rangle =
 \dot{c}_{\mu(\gamma)} \langle \dot{\tau}_i
 \restriction \beta_{\gamma} : i < d \rangle.
\end{equation}
That is, $p'$ forces that the
 $\dot{c}_{\mu(\gamma)}$-color is determined by
 restricting to level $\beta_\gamma$,
 the $\gamma$-th splitting level of
 the $U_i$'s.

Now let $G$ be $\mbb{P}$-generic over $V$
 such that $q' \in G$.
For each $i < d$, let
 $$W_i = \{ \dot{\tau}_G :
 \dot{\tau} \in \mbox{Names}(U_i) \}.$$
By our comments following Theorem~\ref{mainlemma},
 each $W_i$ is a strong subtree of
 $(T_i)_G$.
Since  $q'$ is in  $G$
 and (\ref{mu_eq}) holds,
 the $(\dot{c}_{\mu(\gamma)})_G$-color of a tuple from
 $\bigcup_{\beta \in B}
 \bigoplus_{i < d} W_i(\beta)$
 is determined by restricting to
 the $\gamma$-th splitting level
 $\beta_\gamma$ of the $W_i$'s.
Thus, the conclusion of the modified
 $\mbox{HL}^{tc}$ holds in $V[G]$.
This completes the proof.
\end{proof}

\begin{remark}
Let $\kappa$ be strongly inaccessible.
There is a unique $\kappa$-saturated
 linear order of size $\kappa$
 denoted $\mbb{Q}_\kappa$,
 the $\kappa$-rationals
 \cite{Zhang17}.
%The order type of $\mbb{Q}_\kappa$
% is denoted $\eta_\kappa$.
Zhang proved  in \cite{Zhang17} that
 $\mbox{HL}^{tc}(d,{{<}\kappa},\kappa)$
 implies

\begin{equation}
\label{part_rel}
\begin{pmatrix}
 \mbb{Q}_\kappa \\
 \vdots \\
 \mbb{Q}_\kappa
\end{pmatrix}
\rightarrow
\begin{pmatrix}
 \mbb{Q}_\kappa \\
 \vdots \\
 \mbb{Q}_\kappa
\end{pmatrix}
^{\underbrace{1,...,1}_{d+1}}
_{{<\kappa},(d+1)!}
\end{equation}
This partition relation means
 that given any
 $\sigma < \kappa$ and any coloring
 $c : \prod_{i<{d+1}} \mbb{Q}_\kappa \to \sigma$,
 there is a sequence
 $\langle X_i \subseteq \mbb{Q}_\kappa :
 i < d+1 \rangle$,
 where each $X_i$ is
 order isomorphic to $\mbb{Q}_\kappa$,
 such that $$|c `` \prod_{i < d+1} X_i| \le (d+1)!
$$
Thus,  whenever
 $V$ satisfies
 $\mbox{HL}^{tc}(d,{{<}\kappa},\kappa)$,
combining
Theorem~\ref{hltail_preserved}
 with Zhang's result yields that
every forcing
 extension of $V$ by a poset of size less than $\kappa$
 satisfies the partition relation
 (\ref{part_rel}).

% every forcing extension of $V$ by a poset
% of size $< \kappa$ will satisfy the partition
% relation \ref{part_rel}.
In contrast, other partition relations
 are falsified by small forcings.
For example in \cite{Hajnal/Komjath},
% (see p.18 of Zhang
% for the citations!!!),
 Hajnal and Komj\'{a}th define a fixed poset $\mbb{K}$
 that forces the following
 for sufficiently large $\kappa$:
 $$\mbb{Q}_\kappa
 \not\rightarrow [\mbb{Q}_\kappa]^2_{\omega}.$$
% forces the negation of
% $$\mbb{Q}_\kappa \rightarrow
% (\mbb{Q})^2_{<\kappa,<\omega_1}.$$
This partition relation means that
% any $\sigma < \kappa$ and any coloring
 there exsits a coloring
 $c : [\mbb{Q}_\kappa]^2 \to \omega$
 such that
 for any set $X \subseteq \mbb{Q}_\kappa$
 of order type $\mbb{Q}_\kappa$, we have
 $$f `` [X]^2 = \omega.$$
That is, there is a coloring
 of the pairs from
 $\mbb{Q}_\kappa$ using $\omega$ colors
 such that no set order isomorphic to
 $\mbb{Q}_\kappa$ omits any color.

More specifically, Hajnal and Komj\'{a}th
 show that
 that assuming there are no Suslin trees of height $\omega_1$
 (which can be forced by a small forcing),
 then after adding a Cohen real,
% then after performing any forcing which adds a Cohen real,
%is that $\mbb{K}$ forces that
 there is a linear ordering $\theta$
 of size $\omega_1$ such that
 for any linear ordering
 $\Omega$, there is a coloring
 $c : [\Omega]^2 \to \omega$ such that
 every subset of $\Omega$ order isomorphic to
 $\theta$ does not omit any color.
\end{remark}

%%%%%%%%%%%%%%%%%%%%%%%%%
%%%%%%%%%%%%%%%%%%%%%%%%%
%%%%%%%%%%%%%%%%%%%%%%%%%

\section{Reflection}\label{sec.reflection}

At inaccessible cardinals,
the Halpern-L\"{a}uchli Theorem  reflects.
In  Proposition \ref{prop5.1}, we show   that for $\kappa$ strongly inaccessible,
if $\mbox{SDHL}$
 holds on a stationary set below $\kappa$,
 then it holds at $\kappa$.
In this proposition,
 \mbox{SDHL} \textit{cannot} be replaced by
 $\mbox{HL}^{tc}$, which we will explain
 in the next paragraph.
In Proposition \ref{measurable_prop},
we prove that SDHL holds  at a measurable cardinal
$\kappa$
if and only if the set of ordinals below $\kappa$  where SDHL holds is a member of any normal ultrafilter on $\kappa$.
By Proposition \ref{prop.HL=SDHL},
the same statement holds  for $\mbox{HL}$.
It also holds for $\mbox{HL}^{tc}$.
%Furthermore, the same statement  holds  for the tail cone version.
These two propositions imply Theorem \ref{thm.kclosepres}, that the Halpern-\Lauchli\ Theorem  at a measurable cardinal $\kappa$ is preserved by ${<}\kappa$-closed forcings.

Let us explain why Proposition~\ref{prop5.1}
 does not hold for $\mbox{HL}^{tc}$.
The problem is we could use the argument of
 Theorem~\ref{thm.con-notwc} in the next section to
 get $\mbox{HL}^{tc}$ to hold at a cardinal that is not weakly compact,
 which is impossible by \cite{Zhang17}.
That is,
 assume Proposition~\ref{prop5.1} does hold for $\mbox{HL}^{tc}$
 and start with $V$ satisfying $\mbox{HL}^{tc}$ at
 a measurable $\kappa$.
Then perform any nontrivial forcing of size less than $\kappa$ to  obtain some generic extension $V[G]$.
In $V[G]$,
 $\kappa$ is still measurable and $\mbox{HL}^{tc}$ holds at $\kappa$.
So in $V[G]$, $\mbox{HL}^{tc}$ holds for a stationary set of
 $\lambda < \kappa$.
Now let $V[G][H]$ be any nontrivial forcing extension of $V[G]$
 by a ${<}\kappa$-closed forcing.
Then in $V[G][H]$, $\mbox{HL}^{tc}$ holds for a stationary
 set of $\lambda < \kappa$.
Since we are assuming Proposition~\ref{prop5.1} holds for $\mbox{HL}^{tc}$,
 then in $V[G][H]$, $\mbox{HL}^{tc}$ holds at $\kappa$.
This is impossible, because by a result of Hamkins
 \cite{Hamkins98} any nontrivial forcing of size less than $\kappa$
 followed by any nontrivial ${<}\kappa$-closed forcing
 causes $\kappa$ to not be weakly compact.

\begin{proposition}\label{prop5.1}
\label{stat_prop}
Let $\kappa$ be a cardinal such that either
\begin{itemize}
\item $\kappa$ is strongly inaccessible or
\item $\cf(\kappa) \ge \omega_1$ and
 $\kappa$ is the limit of strongly inaccessible cardinals.
\end{itemize}
Let $1 \le d < \omega$ and
 $1 \le \sigma < \kappa$, and
assume that $\mbox{SDHL}(d, \sigma, \alpha)$ holds
 for a stationary subset of $\kappa$.
Then $\mbox{SDHL}(d, \sigma, \kappa)$ holds.
\end{proposition}

\begin{proof}
%First notice that the assumption that $\mbox{SDHL}(d, \sigma, \alpha)$ holds
% for a stationary set  of $\alpha < \kappa$ implies that
%$\kappa$ is a limit of strongly inaccessible cardinals, since the definition of $\mbox{SDHL}(d, \sigma, \alpha)$  assumes that $\alpha$ is strongly inaccessible.
Let $\langle T_i : i < d \rangle$
 be a sequence of regular $\kappa$-trees
 and let $c : \bigotimes_{i < d} T_i \to \sigma$
 be a coloring.
If we can find an $\alpha < \kappa$ such that each
 $T_i \cap {^{<\alpha}\kappa}$ is
 a regular $\alpha$-tree
 and $\mbox{SDHL}(d, \sigma, \alpha)$ holds,
 then we will be done.
This is because a monochromatic somewhere dense level
 matrix in
 $\bigotimes_{i<d} (T_i \cap {^{<\alpha}\kappa})$,
 is automatically a monochromatic somewhere dense
 level matrix in $\bigotimes_{i<d} T_i$.

Fix $i < d$.
The following standard argument
shows that there is a club
 $C_i \subseteq \kappa$ such that for
  each  $\alpha \in C_i$, the following hold:
\begin{enumerate}
\item
$\alpha$ is a cardinal;
\item
Each level of $T_i \cap {^{<\alpha} \kappa}$ has size less than $\alpha$;
\item
$T_i \cap {^{<\alpha} \kappa}$ is perfect.
\end{enumerate}

Let $\alpha_0=0$ be the least member of $C_i$.
Given $\alpha_{\gamma}$, the $\gamma$-th member   of $C_i$,
construct $\alpha_{\gamma+1}$ as follows.
Let $\lambda_0=\alpha_{\gamma}$.
Given
$\lambda_n$ for $n<\omega$, let $\lambda_{n+1}$ be the least cardinal above $\lambda_n$ such that
both (i) and (ii) hold:
 (i) For each $\beta<\lambda_n$,
$T_i\cap{}^{\beta}\kappa$  is contained in
${}^{\beta}\lambda_{n+1}$ and
has cardinality less than
$ \lambda_{n+1}$.
(ii) For each $t\in T_i\cap {}^{\lambda_n}{\kappa}$,
there are at least two incomparable extensions of $t$ in
$T_i\cap {}^{<\lambda_{n+1}}\kappa$.
Note that $\lambda_{n+1} < \kappa$
 by the cardinal assumption on $\kappa$.
Now define $\alpha_{\gamma+1}= \sup_{n < \omega} \lambda_n$.
Since $\cf(\kappa) \ge \omega_1$
 we have
 $\alpha_{\gamma+1} < \kappa$.
%Since $\kappa$ is inaccessible, $\alpha_{\gamma+1}$ is a cardinal (of cofinality $\omega$) below $\kappa$, and
By the construction, $\alpha_{\gamma+1}$ satisfies (1) - (3).

Given a limit ordinal  $\gamma<\kappa$  and
 the increasing sequence $\langle \alpha_{\zeta}:\zeta<\gamma\rangle$,
% the first $\gamma$ members of $C_i$,
 define $\alpha_{\gamma}=\sup_{\zeta<\gamma}\alpha_{\zeta}$.
Note that $\alpha_\gamma$ automatically satisfies (1) - (3).
% and further, that if $\alpha_\gamma$ is strongly inaccessible, then $T_i\cap {}^{<\alpha_\gamma}\kappa$ is a regular $\alpha_\gamma$-tree.
Thus, given any $\alpha \in C_i$,
 $T_i \cap {^{<\alpha} \kappa}$ is a regular $\alpha$-tree.
This defines $C_i$ as desired, and it is clear that $C_i$ is club.

Let $S$ be a stationary subset of  $\kappa$ such that   SDHL$(d,\sigma, \alpha)$ holds for each $\alpha\in S$.
The set
 $\bigcap_{i < d}  C_i$
 is a club subset of $\kappa$, so it must intersect $S$
 (here we are using that $\cf(\kappa) \ge \omega_1$).
Take any $\alpha < \kappa$ in the intersection.
Then
$\alpha$ is a cardinal,
$T_i \cap {^{<\alpha}\kappa}$ is a regular $\alpha$-tree
 for each $i < d$,
 and SDHL$(d,\sigma,\kappa)$ holds,
 which is what we intended to show.
\end{proof}

For measurable cardinals,
 we have an even stronger form of reflection
 using a normal ultrafilter.
The same argument works for $\mbox{HL}$
 and $\mbox{HL}^{tc}$,
 because these are all statements about
 $V_{\kappa+1}$.

\begin{proposition}\label{measurable_prop}
Let $\kappa$ be a measurable cardinal and
 $\mc{U}$ be a normal ultrafilter on $\kappa$.
Then $$\mbox{SDHL}(d,\sigma,\kappa)
 \Leftrightarrow
 \{ \alpha < \kappa :
 \mbox{SDHL}(d,\sigma,\alpha) \} \in \mc{U}.$$
\end{proposition}

\begin{proof}
Let $j : V \to M$ be the ultrapower embedding
 coming from $\mc{U}$.
Since $V_{\kappa+1} \subseteq M$,
 $$\mbox{SDHL}(d,\sigma,\kappa) \Leftrightarrow
 \mbox{SDHL}(d,\sigma,\kappa)^M.$$
By \L{}os's Theorem,
 $$\mbox{SDHL}(d,\sigma,\kappa)^M \Leftrightarrow
 \{ \alpha < \kappa :
 \mbox{SDHL}(d,\sigma,\alpha) \} \in \mc{U}.$$
\end{proof}

\begin{theorem}\label{thm.kclosepres}
\label{less_then_kappa_closed}
Suppose $\kappa$ is a measurable cardinal.
Let $1 \le d < \omega$ and
 $1 \le \sigma < \kappa$ be given, and
assume $\mbox{SDHL}(d,\sigma,\kappa)$ holds.
If $\mbb{P}$ preserves stationary subsets of $\kappa$
 and adds no new bounded subsets of $\kappa$,
 then $\mbox{SDHL}(d,\sigma,\kappa)$ holds
 after forcing with $\mbb{P}$.
In particular, if
 $\mbb{P}$ is
 ${<}\kappa$-closed,
 then $\mbox{SDHL}(d,\sigma,\kappa)$ holds
 after forcing with $\mbb{P}$.
\end{theorem}

\begin{proof}
As ${<}\kappa$-closed forcings preserve stationary subsets of $\kappa$ and add no new bounded subsets of $\kappa$,
 we need only prove the first half of the theorem.

Fix a normal ultrafilter $\mc{U}$ on $\kappa$.
Since $\mbox{SDHL}(d,\sigma,\kappa)$ holds,
 by Proposition~\ref{measurable_prop}
 the set
 $$S = \{ \alpha < \kappa :
 \mbox{SDHL}(d,\sigma,\alpha) \}$$
 is in $\mc{U}$.
Hence, $S$ is stationary.
Since $\mbb{P}$  preserves stationary
 subsets of $\kappa$,
 $1 \forces \check{S}$ is stationary.
For $\alpha < \kappa$,
 since $\mbox{SDHL}(d,\sigma,\alpha)$
 is a statement about $V_{\alpha+1}$,
 and $V_{\alpha+1}$ is the same when computed
 in the forcing extension by $\mbb{P}$,
 we have that $\mbb{P}$ does not change the truth value of
 $\mbox{SDHL}(d,\sigma,\alpha)$ for any $\alpha < \kappa$.
%Since $\mbb{P}$ adds  no new bounded subsets of $\kappa$,
% then for $\alpha < \kappa$,
% the set of regular $\alpha$-trees in $V$
% is the same as the set of regular $\alpha$-trees in the extension.
%Also, the set of relevant colors is unchanged.
%each regular $\alpha$-tree in the forcing extension is a regular $\alpha$ tree in the ground model.
So,
 \begin{equation}
1 \forces \{ \alpha < \check{\kappa} :
 \mbox{SDHL}(\check{d},\check{\sigma},\alpha) \}
 \mbox{ is stationary}.
\end{equation}
By Proposition~\ref{stat_prop},
 $\mbb{P}$ forces that $\mbox{SDHL}
 (\check{d},\check{\sigma},\check{\kappa})$ holds.

\end{proof}

%%%%%%%%%%%%%%%%%%%%%%%%%%
%%%%%%%%%%%%%%%%%%%%%%%%%%
%%%%%%%%%%%%%%%%%%%%%%%%%%

\section{SDHL at a Cardinal That is Not Weakly Compact}\label{sec.SDHLnotwc}

In \cite{Dobrinen/Hathaway17}, we proved that
 SDHL$(1,k,\kappa)$ holds for all finite $k$
 and all infinite cardinals $\kappa$.
% So, by the equivalence of SDHL and HL for all
%  strongly inaccessible cardinals $\kappa$
%  pointed out to us by Zhang and noted in \cite{Dobrinen/Hathaway17Addendum},
%  HL$(1,k,\kappa)$ holds for every strongly inaccessible
%  $\kappa$ and every finite $k$.
So, by the equivalence of SDHL and HL for all
  strongly inaccessible cardinals $\kappa$,
%  pointed out to us by Zhang and noted in \cite{Dobrinen/Hathaway17Addendum},
  HL$(1,k,\kappa)$ holds for every strongly inaccessible
  $\kappa$ and every finite $k$.
In \cite{Zhang17},
 Zhang showed that this can be improved to
 ${HL}^{asym}(1,\sigma,\kappa)$ holding for \textit{all}
 $\sigma < \kappa$, but he required $\kappa$
 to be weakly compact.
So, it is natural to wonder whether
 $\kappa > \omega$ needs to be weakly compact in order for
 $\mbox{HL}(2,\sigma,\kappa)$
 to hold for \textit{all} $\sigma < \kappa$.

While we were writing \cite{Dobrinen/Hathaway17}
 we discovered the derived tree theorem
 and the proof in this section,
 which answers the question in the negative.
In the meantime,
 Zhang discovered a different proof
 of the consistency of
 $(\forall \sigma < \kappa)\,
 \mbox{HL}(2,\sigma,\kappa)$
 for a $\kappa$ that is \textit{not} weakly compact.
Specifically,
 in Theorem 5.8 of \cite{Zhang17}
 he proved that if for all $d < \omega$,
 $\kappa$ is measurable whenever one adds
 $\kappa^{+d}$ many Cohen subsets of $\kappa$,
 then there is a forcing extension in which $\kappa$
 is inaccessible but not weakly compact,
 and in which HL$(d,\sigma,\kappa)$
 holds for all $d < \omega$
 and all $\sigma < \kappa$.
The theorem we will present now
 implies this, but has a different proof and applies to a broad collection of forcings.

\begin{definition}
For $1 \le d < \omega$
 and an infinite cardinal $\kappa$,
 $\Psi_{d,\kappa}$ is the statement
 $$(\forall \sigma < \kappa)\, \mbox{HL}(d,\sigma,\kappa).$$
\end{definition}

%Recall that Add$(\kappa,\kappa^{+d})$ denotes the  forcing which adds $\kappa^{+d}$ many Cohen subsets of $\kappa$.
In \cite{Dobrinen/Hathaway17} we showed the following:
\begin{theorem}
Let $1 \le d < \omega$.
If $\kappa$ is measurable whenever one adds
 $\kappa^{+d}$ many Cohen subsets of $\kappa$,
 then $\Psi_{d,\kappa}$ holds (in V).
\end{theorem}

%In \cite{Dobrinen/Hathaway17}, we proved that SDHL$(1,k,\kappa)$, equivalently HL$(1,k,\kappa)$,  holds for each positive integer $k$ and weakly compact cardinal $\kappa$.
%Whether or not SDHL could hold at an uncountable cardinal which is not weakly compact was left open.
%Zhang improved this in two directions in \cite{Zhang17}.
%First,  he showed   that the stronger asymmetric  version  $\mbox{HL}^{asym}(1,\delta,\kappa)$ holds at any  weakly compact  cardinal $\kappa$,  where $\delta$ is any ordinal less than  $\kappa$.
%Second, he proved that relative to the existence of  a measurable cardinal,  HL$(1,\delta,\kappa)$ holds for a strongly  inaccessible $\kappa$ which is not weakly compact ($\delta<\kappa$).
%Moreover, in Theorem 5.8 of \cite{Zhang17},
%he proved that  if  for all $d<\omega$, $\kappa$ is a measurable after adding $\kappa^{+d}$ many $\kappa$-Cohen sets,
%then there is a forcing extension in which $\kappa$ is inaccessible but not weakly compact,
%and in which $\mbox{HL}(d,\delta,\kappa)$ holds for all $d<\omega$ and  $\delta<\kappa$.

%Here, we show that starting with a nice enough  model of ZFC with a measurable cardinal $\kappa$,
%any non-trivial forcing of size $<\kappa$
% followed by a non-trivial ${<}\kappa$-closed
% forcing produces a model in which
%$\kappa$ is not weakly compact and
%$ \mbox{HL}(d,\sigma,\kappa)$ holds.

%Recall that Add$(\kappa,\kappa^{+d})$ denotes the  forcing which adds
%$\kappa^{+d}$ many $\kappa$-Cohen sets.

\begin{theorem}\label{thm.con-notwc}
Let $1 \le d < \omega$
 and $\kappa$ be measurable.
Assume $\Psi_{d,\kappa}$ holds.
%Let $1 \le d < \omega$, and
%let $V$ be any model of ZFC in which $\kappa$ is measurable  and remains so  after forcing with
%Add$(\kappa,\kappa^{+d})$.
Then
any non-trivial forcing of size less than $\kappa$
 followed by a non-trivial ${<}\kappa$-closed
 forcing produces a model in which
$\kappa$ is not weakly compact and
 $\Psi_{d,\kappa}$ holds.
%$ \mbox{HL}(d,\sigma,\kappa)$ holds for all $\sigma<\kappa$.
\end{theorem}

\begin{proof}
%Let $\Psi$ be the statement
% $$
%(\forall \sigma < \kappa)\,
% \mbox{SDHL}(d,\sigma,\kappa).
%$$
%Let $V$ be any model of ZFC in which $\kappa$ is measurable  and remains so after  forcing to add
%$\kappa^{+d}$ many $\kappa$-Cohen sets.
%For instance, such a model appears in the proof of Theorem 4.6 of
% \cite{Dobrinen/Hathaway17}.
%Then in $V$,
% $\Psi$ holds.
By a theorem of Hamkins
 \cite{Hamkins98},
 any non-trivial forcing of size less than $\kappa$
 followed by a non-trivial ${<}\kappa$-closed
 forcing will force $\kappa$  to \textit{not} be weakly compact.
%In particular, $\kappa$ is not weakly compact in $V[G]$.

Let $\mbb{P}$ be  any non-trivial forcing of size $<\kappa$.
Let $G$ be $\mbb{P}$-generic over $V$.
%Then  any generic extension
%  $V[G]$ of $V$ by $\mbb{P}$
%  forced by $\mbb{P}$ over $V$,
Then $\Psi_{d,\kappa}$ holds in $V[G]$
% preserves $\Psi_{d,\kappa}$
 by Theorem~\ref{sdhl_is_preserved}.
Let $\mbb{Q}$ be any non-trivial ${<}\kappa$-closed forcing
 in $V[G]$, and let $H$ be $\mbb{Q}$-generic over $V[G]$.
%Since
%any non-trivial ${<}\kappa$-closed forcing
% over $V[G]$   preserves that
%$\kappa$ is not weakly compact,
%$\kappa$ is not weakly compact in $V[G][H]$.
By Hamkins's result, $\kappa$ is not weakly compact
 in $V[G][H]$.
Since $\Psi_{d,\kappa}$ holds in
 $V[G]$
 and $\kappa$ is measurable in this model, it follows from
 Theorem~\ref{less_then_kappa_closed}
 that
 $\Psi_{d,\kappa}$ also holds in
 $V[G][H]$.
\end{proof}

%%%%%%%%%%%%%%%%%%%5555
%%%%%%%%%%%%%%%%%%%555555
%%%%%%%%%%%%%%%%%%%%

\section{Open Problems}\label{conclusion}

The main open problem concerning the
 Halpern-L\"auchli Theorem at uncountable cardinals
 is the following:
\begin{question}
Is it consistent for
 $\mbox{HL}(d,\sigma,\kappa)$ to fail for some uncountable cardinal $\kappa$?
\end{question}

Because this is unanswered,
 there are many secondary questions.
For example,
 even though $\mbox{HL}(d,\sigma,\kappa)$  does not imply $\kappa$ itself
 must  be weakly compact,
 does it have any large cardinal
 strength?
Does HL have so much large cardinal
 strength that it cannot hold in L; or
 does HL always hold in L?
Does  the existence of  say $0^\#$  imply
 that within $L$,
 HL holds at some, or all, strongly inaccessible cardinals?

In  all known models  in which $\mbox{HL}(d,\sigma,\kappa)$ holds for some strongly inaccessible $\kappa$, GCH fails.
Such models appear in \cite{Shelah91},
\cite{Dzamonja/Larson/MitchellQ09},
\cite{Dzamonja/Larson/MitchellRado09},
\cite{Dobrinen/Hathaway17},
\cite{Zhang17} and the preceding sections of this article.
Is $\mbox{HL}(d,\sigma,\kappa)$ for $\kappa$ strongly inaccessible  consistent with GCH?

In this article, we showed that various forms of HL are preserved by small forcings or by ${<}\kappa$-closed forcings.
What other types of forcings preserve HL?
An obvious question is the following:
\begin{question}
Do $\kappa$-c.c.\ forcings preserve HL$(d,\sigma,\kappa)$, for $0<d<\omega$ and $0<\sigma<\kappa$?
\end{question}

Many variants of these questions can be formulated, and progress on any of them will lead to a better understanding of Halpern-L\"{a}uchli Theorems and associated partition relations on uncountable structures.

\bibliographystyle{amsplain}
\bibliography{DobrinenHathawayHL2_may4_2019}

\providecommand{\bysame}{\leavevmode\hbox to3em{\hrulefill}\thinspace}
\providecommand{\MR}{\relax\ifhmode\unskip\space\fi MR }
\providecommand{\MRhref}[2]{%
  \href{http://www.ams.org/mathscinet-getitem?mr=#1}{#2}
}
\providecommand{\href}[2]{#2}
\begin{thebibliography}{10}

\bibitem{DevlinThesis}
Dennis Devlin, \emph{Some partition theorems for ultrafilters on $\omega$},
  Ph.D. thesis, Dartmouth College, 1979.

\bibitem{Dobrinen/Hathaway17}
Natasha Dobrinen and Daniel Hathaway, \emph{The {H}alpern-{L}{\"{a}}uchli
  {T}heorem at a measurable cardinal}, Journal of Symbolic Logic \textbf{82}
  (2017), no.~4, 1560--1575.

%\bibitem{Dobrinen/Hathaway17Addendum}
%\bysame, \emph{Addendum to the {H}alpern-{L}{\"{a}}uchli {T}heorem at a
%  measurable cardinal}, Journal of Symbolic Logic (2018), To appear.

\bibitem{Dodos/KanBK}
Pandelis Dodos and Vassilis Kanellopoulos, \emph{{R}amsey {T}heory for
  {P}roduct {S}paces}, American Mathematical Society, 2016.

\bibitem{Dzamonja/Larson/MitchellQ09}
M.~D{\v{z}}amonja, J.~Larson, and W.~J. Mitchell, \emph{A partition theorem for
  a large dense linear order}, Israel Journal of Mathematics \textbf{171}
  (2009), 237--284.

\bibitem{Dzamonja/Larson/MitchellRado09}
\bysame, \emph{Partitions of large {R}ado graphs}, Archive for Mathematical
  Logic \textbf{48} (2009), no.~6, 579--606.

\bibitem{Hajnal/Komjath}
A.~Hajnal and Komj\'{a}th, \emph{A strongly non-{R}amsey order type},
  Combinatorica \textbf{17} (1997), no.~3.

\bibitem{Halpern/Lauchli66}
J.~D. Halpern and H~L{\"{a}}uchli, \emph{A partition theorem}, Transactions of
  the American Mathematical Society \textbf{124} (1966), 360--367.

\bibitem{Halpern/Levy71}
J.~D. Halpern and A.~L{\'{e}}vy, \emph{The {B}oolean prime ideal theorem does
  not imply the axiom of choice}, Axiomatic Set Theory, Proc. Sympos. Pure
  Math., Vol. XIII, Part I, Univ. California, Los Angeles, Calif., 1967,
  American Mathematical Society, 1971, pp.~83--134.

\bibitem{Hamkins98}
Joel~David Hamkins, \emph{Small forcing makes any cardinal superdestructible},
  Journal of Symbolic Logic \textbf{63} (1998), 51--58.

\bibitem{Laflamme/Sauer/Vuksanovic06}
Claude Laflamme, Norbert Sauer, and Vojkan Vuksanovic, \emph{Canonical
  partitions of universal structures}, Combinatorica \textbf{26} (2006), no.~2,
  183--205.

\bibitem{Laver84}
Richard Laver, \emph{Products of infinitely many perfect trees}, Journal of the
  London Mathematical Society (2) \textbf{29} (1984), no.~3, 385--396.

\bibitem{Milliken79}
Keith~R. Milliken, \emph{A {R}amsey theorem for trees}, Journal of
  Combinatorial Theory, Series A \textbf{26} (1979), 215--237.

\bibitem{Sauer06}
Norbert Sauer, \emph{Coloring subgraphs of the {R}ado graph}, Combinatorica
  \textbf{26} (2006), no.~2, 231--253.

\bibitem{Shelah91}
Saharon Shelah, \emph{Strong partition relations below the power set:
  consistency -- was {S}ierpinski right? ii}, Sets, Graphs and Numbers
  (Budapest, 1991), vol.~60, Colloq. Math. Soc. J{\'{a}}nos Bolyai,
  North-Holland, 1991, pp.~637--688.

\bibitem{TodorcevicBK10}
Stevo Todorcevic, \emph{Introduction to {R}amsey {S}paces}, Princeton
  University Press, 2010.

\bibitem{Farah/TodorcevicBK}
Stevo Todorcevic and Ilijas Farah, \emph{Some {A}pplications of the {M}ethod of
  {F}orcing}, Yenisei Series in Pure and Applied Mathematics, 1995.

\bibitem{Zhang17}
%Jing Zhang, \emph{A tail cone version of the {H}alpern-{L}{\"{a}}uchli
%  {T}heorem at a large cardinal}, Journal of Symbolic Logic, 23 pp, To appear.
Jing Zhang, \emph{A tail cone version of the {H}alpern-{L}{\"{a}}uchli Theorem at a large cardinal.}
The Journal of Symbolic Logic, 1-23. doi:10.1017/jsl.2017.55.

\end{thebibliography}

\end{document}